\documentclass[11pt]{amsart}

\usepackage{amsmath,amsfonts,amsthm,graphicx,fullpage}
\usepackage[foot]{amsaddr}
\usepackage[hidelinks]{hyperref}

\newcommand{\cF}{\mathcal{F}}
\newcommand{\cG}{\mathcal{G}}

\newcommand{\cL}{\mathcal{L}}
\newcommand{\cP}{\mathcal{P}}

\newcommand{\floor}[1]{\left\lfloor {#1} \right\rfloor}

\DeclareMathOperator{\E}{E}
\DeclareMathOperator{\Var}{Var}
\DeclareMathOperator{\Uniform}{Uniform}

\allowdisplaybreaks 

\newtheorem{theorem}{Theorem}
\newtheorem{conjecture}{Conjecture}
\newtheorem{prop}{Proposition}
\newtheorem{lemma}{Lemma}[section]

\title{Hamiltonian increasing paths in random edge orderings}

\author{Mikhail Lavrov}
\address{Department of Mathematical Sciences, Carnegie Mellon University, Pittsburgh, PA 15213}
\email{mlavrov@andrew.cmu.edu}

\author{Po-shen Loh}
\email{ploh@cmu.edu}
\thanks{Research supported in part by NSF grant DMS-1201380, an NSA Young Investigators Grant, and by a USA-Israel BSF Grant.}

\begin{document}

\begin{abstract}
  Let $f$ be an edge ordering of $K_n$: a bijection $E(K_n) \to
  \{1,2,\dots, {n\choose 2}\}$. For an edge $e \in E(K_n)$, we call $f(e)$
  the label of $e$. An \emph{increasing path} in $K_n$ is a simple path
  (visiting each vertex at most once) such that the label on
  each edge is greater than the label on the previous edge. We let $S(f)$
  be the number of edges in the longest increasing path.  Chv\'atal and
  Koml\'os raised the question of estimating $m(n)$: the minimum value of
  $S(f)$ over all orderings $f$ of $K_n$. The best known bounds on $m(n)$
  are $\sqrt{n-1} \le m(n) \le
  \left({\textstyle\frac12} + o(1)\right)n$, due respectively to Graham and
  Kleitman, and to Calderbank, Chung, and Sturtevant.  Although the problem
  is natural, it has seen essentially no progress for three decades.

  In this paper, we consider the average case, when the ordering is chosen
  uniformly at random. We discover the surprising result that in the random
  setting, $S(f)$ often takes its maximum possible value of $n-1$ (visiting
  all of the vertices with a Hamiltonian increasing path).  We prove that
  this occurs with probability at least about $1/e$.  We also prove that
  with probability $1-o(1)$, there is an increasing path of length at least
  $0.85n$, suggesting that this Hamiltonian (or near-Hamiltonian)
  phenomenon may hold asymptotically almost surely.
\end{abstract}

\maketitle

\section{Introduction}

The classical result of Erd\H{o}s~and~Szekeres~\cite{erdos35} states that
any permutation of ${\{1, 2, \dots, n^2+1\}}$ contains a monotonic
subsequence of length $n+1$. Many extensions have been found for this
theorem: see, e.g., any of \cite{FoxPachSudakovSuk, Kalmanson,
MoshkovitzShapira, steele95, SzaboTardos}.  In this paper, we consider the
direction started by Chv\'atal and Koml\'os~\cite{chvatal71}.  They posed
the natural analogue of the problem for walks in a graph, which may be
considered an extension of Erd\H{o}s--Szekeres in a similar spirit to how
Ramsey's theorem is an extension of the pigeonhole principle.  Rather than
order the integers $\{1,2,\dots,n\}$, we order the edges of $K_n$ by
setting a bijection $f: E(K_n) \to \{1, 2, \dots, {n\choose2}\}$.  A walk
in $K_n$ whose edges are $(e_1, e_2, \dots, e_k)$ is called $f$-increasing
if the labels $f(e_1), f(e_2), \dots, f(e_k)$ form an increasing sequence.
(In this setting, we can assume that the monotone sequence of labels is
increasing without loss of generality, since a decreasing walk is just an
increasing walk traversed backwards.) As in Erd\H{o}s--Szekeres, the
objective is to prove a worst-case lower bound on the length of the longest
increasing walk.  Here, a walk is permitted to visit the same vertex
multiple times.

This question was resolved by Graham and Kleitman~\cite{graham73}. In
\cite{winkler08b}, Winkler communicates an elegant formulation of their
solution, which is due to Friedgut: a pedestrian stands at every vertex of
$K_n$, and the edges are called out in increasing order; whenever an edge
is called out, the pedestrians on its endpoints switch places. After all
edges have been called out, the $n$ pedestrians have taken a total of
$n(n-1)$ steps, and therefore at least one must have taken at least $n-1$
steps, producing an increasing walk with length (number of edges) at least
$n-1$.

This is easily seen to be tight for even $n$, for which $K_n$ can be
partitioned into $n-1$ perfect matchings, and edges within each individual
matching can receive consecutive labels. For odd $n$, a partition into $n$
maximal matchings only gives an upper bound of $n$; a more complicated
argument in~\cite{graham73} shows that $n-1$ is still correct for all $n$
except $n=3$ and $n=5$, where $n$ is the right answer.

Chv\'atal and Koml\'os also posed the corresponding problem for
self-avoiding walks, or paths, which are not permitted to revisit any
vertex.  Self-avoiding walks are generally much harder to analyze, and
indeed, in this setting, even determining the answer asymptotically is
still an open question. Calderbank, Chung, and
Sturtevant~\cite{calderbank84} construct an ordering of $K_n$ for which no
increasing path is longer than $\left(\frac12+o(1)\right)n$. The best lower
bound known was also proven by Graham and Kleitman in~\cite{graham73},
where they show that there must always be an increasing path of length
$\sqrt{n-1}$.

A simple variant of the ``pedestrian argument'' establishes this lower bound, which we include
for completeness. Suppose we modify the pedestrians' algorithm so that
if either pedestrian would visit an already-visited vertex, instead both
pedestrians stay put. If all pedestrians take at most $k$ steps, then at
most $\frac12 nk$ edges are walked; each pedestrian can refuse at most ${k
\choose 2}$ edges, so at most $n {k \choose 2}$ edges are refused, for a
total of at most $\frac12 n k^2$ edges. Since all ${n \choose 2}$ edges are
either walked or refused, $\frac12 nk^2 \ge {n \choose 2}$, so $k \ge
\sqrt{n-1}$.

Many extremal questions for combinatorial structures have also been studied
in the random setting (see, e.g., either of the books \cite{bollobas85,
janson00} on random graphs), which is in a sense equivalent to asking about
the average-case rather than the worst-case behavior of some property.  In
many situations, one can prove that interesting properties hold with
probability $1-o(1)$ over the space of random objects, in which case the
property is said to hold \emph{asymptotically almost surely}, or a.a.s.\
for short.  For example, the random analogue of the Erd\H{o}s--Szekeres
result considers the length $I_n$ of the longest increasing subsequence in
a random permutation of $\{1, 2, \dots, n\}$, and this is a well-studied
topic: it is known \cite{logan77, versik77} that $I_n \sim 2\sqrt{n}$
a.a.s.  (Here and in the remainder, we write $X \sim Y$ to denote $\lim_{n
\rightarrow \infty} \frac{X}{Y} = 1$.)

In this paper we consider the random version of the increasing path
problem.  Suppose the ordering $f$ is chosen uniformly at random. What can
we say about the length of the longest $f$-increasing path?  It is natural
to begin by considering the performance of the greedy algorithm on a
randomly ordered graph, since all walks traced by pedestrians in the above
argument are greedy in the following sense: every walk exits each vertex
along the minimally-labeled edge which maintains the increasing property.

\begin{prop}
  \label{random-greedy}
  Let $v_0$ be an arbitrary vertex in $K_n$.  Given an edge ordering $f$ of
  the edges of $K_n$, let the greedy $f$-increasing path from $v_0$ be the
  path $v_0 v_1 v_2 \ldots v_t$ with the following properties: (i) $v_0
  v_1$ is the lowest-labeled edge incident to $v_0$, (ii) for each $1 \leq
  i \leq t-1$, $v_{i+1}$ is the vertex $x$ which minimizes the label of
  $v_i x$ over all $x \not \in \{v_0, v_1, \ldots, v_i\}$ with $v_i x$
  exceeding the label of $v_{i-1} v_i$, and (iii) every vertex $x \not \in
  \{v_0, \ldots, v_t\}$ has $v_t x$ labeled below $v_{t-1} v_t$.  Then, if
  $f$ is chosen uniformly at random, the length of the greedy
  $f$-increasing path from $v_0$ is $(1 - \frac1e + o(1))n$ a.a.s.
\end{prop}

Since the analysis in the previous result is tight, one must consider more
complex algorithms in order to find longer paths in the random setting.
The main challenge in analyzing more sophisticated algorithms arises from
the fact that randomness is revealed during the algorithm's execution.  We
introduce a novel extension of the greedy algorithm which adds some
foresight, but which is formulated in a way that is amenable to analysis.
At each step, this \emph{$k$-greedy algorithm} greedily finds a tree of $k$
potential edges that can extend the increasing path, before choosing the
one that has the best short-term prospects.  (The ordinary greedy algorithm
is the $k=1$ case of this algorithm.)  A detailed specification of this
algorithm appears in Section \ref{sec:k-greedy}.

The performance of the $k$-greedy algorithm is related to statistics which
arise in the Chinese Restaurant Process, or equivalently, to the random
variable $L_k$ which measures the length of the longest cycle in a
uniformly random permutation of $\{1, \ldots, k\}$.  Let $\alpha_k = \E[
  \frac1{L_k} + \frac1{L_k +1} + \cdots + \frac1k]$.

\begin{theorem}
  \label{k-greedy}
  If an edge ordering $f$ of $K_n$ is chosen uniformly at random, then the
  $k$-greedy algorithm finds an $f$-increasing path of length $(1 -
  e^{-1/\alpha_k} + o(1)) n$ a.a.s.  Also, $\alpha_k$ is monotone
  decreasing in $k$ and explicitly computable.  The particular choice
  $k=100$ produces an increasing path of length $0.85n$ a.a.s.
\end{theorem}

\noindent \textbf{Remark.} As $k \to \infty$, the monotonicity in the above
result implies that $\alpha_k$ converges to $\alpha = \lim_{k \rightarrow
\infty} \E\big[ - \log \frac{L_k}{k}\big]$, a constant related to the
Golomb--Dickman constant $\lim_{k \rightarrow \infty}
\E\big[\frac{L_k}{k}\big] \approx 0.6243$.  Numerically, we estimate
$\alpha \approx 0.5219$, and $1 - e^{-1/0.5219} \approx 0.853$, so $k=100$
appears to be a near-optimal choice.  

\medskip

The first two results establish successively stronger linear lower bounds
on the increasing path length in a random edge ordering. There is a trivial
upper bound of $n-1$: the length of a Hamiltonian path, and at first
glance, one may assume that a Hamiltonian increasing path would be too much
to hope for.  Indeed, when one calculates the expected number of
$f$-increasing Hamiltonian paths, the total number of paths, which is $n!$,
is almost exactly canceled by the probability that each is increasing,
which is $\frac1{(n-1)!}$.  Thus, the expected number of increasing
Hamiltonian paths is only $n$, which although it tends to infinity, grows
extremely slowly.  In comparison, in the Erd\H{o}s--R\'enyi model $G_{n,p}$,
where each edge appears independently with probability $p$, the expected
number of Hamilton paths is about $n$ when $p \sim \frac{e}{n}$, but
Hamiltonian paths don't appear until $p \sim \frac{\log n}{n}$.
Furthermore, for Hamilton cycles in the random graph process, at the moment
Hamiltonicity is achieved, the number of Hamilton cycles jumps from 0 to
$\big[(1 +o(1))\frac{\log n}{e} \big]^n$, as shown recently by Glebov and
Krivelevich \cite{GlebovKrivelevich} (improving an earlier result of Cooper
and Frieze \cite{CooperFrieze}).  It may therefore come as a surprise that
random edge orderings often have Hamiltonian paths, despite the extremely
low expected value.

\begin{theorem}
  \label{main-theorem}
  If an edge ordering $f$ of $K_n$ is chosen uniformly at random, then an
  $f$-increasing Hamiltonian path exists in $K_n$ with probability at least
  $\frac1e + o(1)$.
\end{theorem}

Theorems \ref{k-greedy} and \ref{main-theorem} complement each other, as
they establish a.a.s.\ almost-Hamiltonicity and almost-a.a.s.\ Hamiltonicity.
Numerical simulations seem to indicate that a stronger result is true,
which we propose as follows.

\begin{conjecture}
  \label{conjecture}
  If an edge ordering $f$ of $K_n$ is chosen uniformly at random, then an
  $f$-increasing Hamiltonian path exists in $K_n$ a.a.s.
\end{conjecture}

The proof of Theorem~\ref{main-theorem} uses the second moment method. Let
$H_n$ be the number of increasing Hamiltonian paths.  As mentioned above,
it is easy to see that $\E[H_n] = n!  \cdot \frac1{(n-1)!} = n$.  The main
step of our proof is to upper-bound the second moment $\E[H_n^2]$.  We
actually go one step further, and asymptotically determine $\E[H_n^2] =
(1+o(1)) e n^2$, from which the result follows.  This asymptotic
computation of $\E[H_n^2]$, up to multiplicative error $1+o(1)$, may be
useful for a full proof of Conjecture \ref{conjecture} via analysis of
variance. For example, the problem might be tractable by the small subgraph
conditioning method first used by Robinson and Wormald~\cite{robinson94} to
prove that the random $d$-regular graph $G_{n,d}$, for $d\ge 3$, contains a
Hamiltonian cycle a.a.s., and this method requires the precise second
moment estimate that we provide.

\section{The length of the greedy increasing path}
\label{sec:greedy}

As a warm-up for the $k$-greedy algorithm, we begin by proving
Proposition~\ref{random-greedy}, which establishes that the greedy
algorithm produces an increasing path of linear length a.a.s.  In this
section, it is convenient to introduce a different (but equivalent) model
for generating the edge labels, which features more independence.  In order
to sample a uniform permutation of $\{1, 2, \ldots, \binom{n}{2}\}$ for the
labels, we choose a labeling $f: E(K_n) \to [0,1]$, where the labels $f(e)$
are i.i.d.\ $\Uniform(0,1)$ random variables. Since with probability $1$ no
two labels will be equal, this induces a total ordering on the edges, and
by symmetry, all orderings occur with uniform probability.

Let $(e_1, e_2, \dots, e_k)$ be the edges of any path, not necessarily
$f$-increasing. We define the \emph{jumps} $X_1, \dots, X_k$ along this
path by $X_1 = f(e_1)$, and
\[
	X_k = (f(e_k) - f(e_{k-1}))\bmod1 =
	\begin{cases}
		f(e_k) - f(e_{k-1}), & \mbox{if }f(e_k) > f(e_{k-1}) \\
		1 + f(e_k) - f(e_{k-1}), & \mbox{if } f(e_k) \le f(e_{k-1}).
	\end{cases}
\]
The sum $X_1 + X_2 + \cdots + X_k$ telescopes to $f(e_k) + p$, where $p$ is the number of points at which the path fails to be $f$-increasing. Therefore the path is $f$-increasing if and only if $X_1 + X_2 + \cdots + X_k \le 1$.

Choose a Hamiltonian path $(e_1, e_2, \dots, e_{n-1})$ by the following
rule: starting from an arbitrary vertex, always take the edge with the
smallest jump. The result coincides with the greedy path for the entire
length of the greedy path. However, when the greedy path would stop, this
rule merely makes a step that isn't $f$-increasing. This allows us to keep
going until $n-1$ edges are chosen, no matter what. The length of the
greedy increasing path will therefore be the largest $k$ for which the
initial segment $(e_1, e_2, \dots, e_k)$ forms an $f$-increasing path;
equivalently, the largest $k$ such that $X_1 + X_2 + \cdots + X_k \le 1$.

When constructing this path, we only expose the labels of the edges as we
encounter them.  Specifically, if we already have the partial Hamiltonian
path $(e_1, \dots, e_{k-1})$, then the next edge $e_k$ will be one of the
$n-k$ edges from the last vertex to a new vertex not on the current path.
Call those possible edges $e_k^{1}, \dots, e_k^{n-k}$, and expose
$f(e_k^1), \dots, f(e_k^{n-k})$.  The jump $X_k$ is given by
\[
	X_k = \min \left\{ (f(e_k^1) - f(e_{k-1}))\bmod 1, \dots, (f(e_k^{n-k}) - f(e_{k-1}))\bmod 1\right\}.
\]
The values $(f(e_k^j) - f(e_{k-1})) \bmod 1$ are also uniformly distributed
on [0,1], and exposing them is equivalent to exposing $f(e_k^1), \dots,
f(e_k^{n-k})$. This means that $X_k$ is the minimum of $n-k$ uniform
random variables, which are independent from each other and all previously
exposed values.

Since $\Pr[X_k \ge x] = (1-x)^{n-k}$, the probability density function is $(n-k) (1-x)^{n-k-1}$, so:
\begin{align*}
	\E[X_k] &= \int_0^1 x (n-k) (1-x)^{n-k-1}\,dx = \frac1{n-k+1}, \\
	\E[X_k^2] &= \int_0^1 x^2 (n-k) (1-x)^{n-k-1}\,dx = \frac2{(n-k+1)(n-k+2)}.
\end{align*}

Suppose $t = \tau n$ for some constant $0 < \tau < 1$, and let $S_t = X_1 +
X_2 + \cdots + X_t$. Then $\E[S_t] = \frac{1}{n} + \cdots +
\frac{1}{n-t+1}$, and therefore $\E[S_t] = \log \frac{n}{n-t} + O(n^{-1}) =
\log \frac{1}{1-\tau} + O(n^{-1})$. Furthermore, we have $\Var[X_k] =
O(n^{-2})$ for $1 \le k \le t$, so $\Var[S_t] = O(n^{-1})$, which means
that $S_t = (1+o(1)) \E[S_t]$ a.a.s.  When $\tau < 1 - \frac1e$,
$\E[S_t]<1$, so $S_t < 1$ a.a.s., and the greedy path is $f$-increasing
through the first $t$ steps. On the other hand, when $\tau > 1-\frac1e$,
$\E[S_t]>1$, so $S_t > 1$ a.a.s., and the greedy path is not $f$-increasing
after the first $t$ steps.  This completes the proof of Proposition
\ref{random-greedy}.  \hfill $\Box$

\section{The $k$-greedy algorithm}

Throughout this section, let $k$ be a constant.  The $k$-greedy algorithm
extends the greedy algorithm by adding some limited look-ahead to the
choice of each edge.  We analyze it using the same model as in the previous
section: each edge receives an independent random label from $\Uniform(0,
1)$. For the purposes of intuition, we think of the label $f(e)$ of an edge
$e$ as the time at which $e$ appears in the graph.  We first describe how
the algorithm would run, given a (deterministic) full labeling of the edges
of $K_n$ with real numbers from $[0, 1]$.  

The challenge with any complex algorithm is dependency between iterations.
Our main innovation is to distill the algorithm and pose it in a way that
is particularly clean and amenable to analysis.

\subsection{\boldmath Algorithm $k$-greedy}
\label{sec:k-greedy}

\begin{enumerate}

  \item Initialize the path $P$ to be a single (arbitrary) vertex $v_1$.
    Initialize the rooted tree $T$ to be the 1-vertex tree with $v_1$ as
    the root.  Initialize the time $\tau$ to be 0.
    
  \item 
    \label{alg:repeat}
    While the rooted tree $T$ has fewer than $k$ edges, do:

    \begin{enumerate}
      \item Let $S$ be the set of all edges with one endpoint in $T$,
        the other endpoint not in $P \cup T$, and label at least $\tau$.

      \item If $S$ is empty, then terminate the algorithm.

      \item Identify the edge of $S$ with minimum label, add it to
        $T$, and set $\tau$ to be its label.
    \end{enumerate}

  \item The rooted tree $T$ now has exactly $k$ edges.  Among the
    children of the root, identify the child $x$ whose subtree (rooted at
    $x$) is the largest.  Extend $P$ by one edge to $x$, and set $T$ to be
    the subtree rooted at $x$.  This may substantially reduce the size of
    $T$, as it deletes all of the other subtrees, as well as the root.

  \item Go back to step (\ref{alg:repeat}).

\end{enumerate}

\begin{lemma}
  The path $P$ produced by the $k$-greedy algorithm is always a simple
  increasing path.
  \label{lem:k-greedy-sound}
\end{lemma}

\begin{proof}
  Consider any moment at which an edge $e$ is added to $T$.  Suppose that
  $e = xy$, where $x$ was previously in $T$, and $y$ is a new leaf of $T$.
  By construction, the label of $e$ is at least $\tau$, but the label of
  every other edge in $P \cup T$ is at most $\tau$ (and a.s.\ not equal to
  $\tau$).  So, by induction, at all times during the algorithm, all paths
  from the first vertex $v_1 \in P$ to any leaf of $T$ are increasing
  paths.  They are all simple paths because we only extend $T$ by edges to
  vertices not currently in $P \cup T$.
\end{proof}

\subsection{Managing revelation of randomness}

\label{sec:manage-randomness}

We now take a closer look at what information needs to be revealed at each
step in order to run the algorithm.  We find that Step (2b) requires a
yes/no answer, and Step (2c) requires the identification of a single edge,
together with its label.  Therefore, if we have access to an oracle which
reports this information upon request, we will be able to run the complete
algorithm.

The information revelation in Step (2b) is a minor issue which we can
easily sidestep.  We accomplish this via fictitious continuation: let the
oracle always answer that ``$S$ is nonempty'', and in the event that $S$ is
indeed empty, it will increase $\tau$ to 1, and return an arbitrary edge
from a leaf of $T$ to a vertex not in $P \cup T$ in Step (2c).  This is a
failure state.  In our analysis below, we will run the algorithm for a
predetermined number of steps, and show that with high probability, $\tau$
is still bounded below 1, because the likelihood of a failure in Step (2b)
is highly unlikely.

We will carefully describe how we manage the exposures in Step (2c).  At
the beginning, the labels on all edges are independent, and each is
uniformly distributed in $[0, 1]$.  Consider the exposure the first time
Step (2c) is encountered.  The oracle reports an edge $v_1 x$ and its
label, and so at this point, the label of $v_1 x$ is certainly determined.
Let us refer to it as $\tau_1$.  We also learn some information about all
other edges $v_1 y$, with $y \not \in \{v_1, x\}$: their labels are not in
the range $[0, \tau_1)$.  We do not learn any restrictions on any the
labels of any other edges.  Importantly, the labels on all edges are
still independent, and uniformly distributed over their
(possibly-restricted) ranges.  They just are not identically distributed.
  
Consider the second time the algorithm encounters Step (2c).  Now, the
oracle reports another edge, say $xz$, and suppose its label is $\tau_2$.
Then, we know that all edges $v_1 y$ and $xy$ with $y \not \in \{v_1, x,
z\}$ have labels outside of the interval $(\tau_1, \tau_2]$.  We already
learned that some of those edges had labels outside $[0, \tau_1)$; for
those, we now know that their labels avoid $[0, \tau_2)$.  Again, all
  labels are still independent and uniformly distributed over their ranges.

So, at every intermediate time $\tau$ during the course of the algorithm,
some edges will have their labels determined, but independence between
non-determined labels is preserved throughout.  Each non-determined label
is still uniformly distributed over some range of the form $[0, 1]
\setminus (I_1 \cup I_2 \cup \ldots \cup I_k)$, where the $I_i$ are
disjoint sub-intervals.  Note that all ranges still completely include
$[\tau, 1]$, and so this phenomenon actually works in our favor, because it
increases the likelihood that we can still use the edge: Step (2a) queries
only edges with labels at least $\tau$.  Since all non-determined edges
have label ranges including $[\tau, 1]$, they satisfy this property with
probability exactly $(1-\tau) / \mu$, where $\mu$ is the measure of their
current range.  This is clearly worst when $\mu = 1$, which corresponds to
the fully unrestricted $[0, 1]$ range.

\subsection{Intuitive calculation}
\label{sec:intuitive}

Now that we have a clean model which definitively indicates how much
randomness is surrendered at each step, we conduct a rough analysis which
captures the main structure of the argument.  This will also derive the
constant in Theorem \ref{k-greedy}.  The key statistic to estimate is the
typical \emph{waiting time}, which we define to be the difference between
the labels of successive edges added to $T$ by Step (2c) of the algorithm.

Suppose that at some stage of the algorithm, the path has length $\ell$,
and the tree $T$ has $j$ vertices.  If all of the $j (n-\ell-j)$ edges
between $T$ and vertices outside $P \cup T$ still had labels which were
uniformly distributed over $[0, 1]$, then the waiting time would be the
minimum of that many random variables $\Uniform(0, 1)$.  The waiting time
would then be exactly $\frac1{j(n-\ell-j) + 1}$ in expectation.  In our
situation, this is not exactly true.  We still have independence, but the
labels are distributed uniformly over sub-ranges of $[0, 1]$.  Also, some
of those $j(n-\ell-j)$ edges could potentially already have their labels
determined, if they had previously been added as edges to $T$ in an earlier
stage of the algorithm, but were discarded by some Step (3).  As mentioned
at the end of Section \ref{sec:manage-randomness}, the first issue is in
our favor, because it only reduces the waiting time.  The second issue
works against us, because it reduces the number of independent random
labels that are competing in Step (2c), but we will show in Section
\ref{sec:k-greedy-rigorous} that in fact both of these effects are
negligible.  So, we first analyze the (fictitious) ideal case.

For now, let us proceed using $\frac1{j(n-\ell-j) + 1} \sim
\frac{1}{n-\ell} \cdot \frac{1}{j}$ as the expected waiting time.  Then,
the expected time until the search tree fills up from $j-1$ to $k$ edges is
asymptotically
\[
	\frac1{n-\ell} \cdot \left( \frac1j + \frac1{j+1} + \cdots +
        \frac1{k}\right) \sim \frac1{n-\ell} \log \frac kj.
\]
At this point, the increasing path is extended by 1, and the search tree
shrinks to the largest subtree determined by a child of the root. From the
formula above, we see that only the size of this subtree is of interest.

In our ideal setting, when a potential edge is added to a search tree with
$j$ vertices, its endpoint in the search tree is randomly distributed
uniformly over all $j$ vertices currently in the tree.  From the point of
view of subtree sizes among children of the root, it therefore starts a new
subtree with probability $\frac1j$ (if its tree-endpoint is the root
itself), or is added to one of the existing children's subtrees with
probability proportional to current size (depending on which subtree its
tree-endpoint is in).  This is equivalent to the Chinese restaurant
process, which generates the cycle decomposition of a uniformly random
permutation: if $\pi$ is a uniformly random permutation of $\{1, \dots,
j-1\}$, then we can transform $\pi$ into a uniformly random permutation of
$\{1, \dots, j\}$ by making $j$ a fixed point with probability $\frac1j$,
and otherwise inserting $j$ in a uniformly chosen point in any cycle (which
means a cycle of length $i$ is chosen with probability $\frac ij$).
Therefore the number of vertices in the largest subtree has the same
distribution as a well-studied random variable: the length $L_k$ of the
longest cycle in a uniformly random permutation of $\{1, \dots, k\}$.

Define the random variable $A_k = \frac1{L_k} + \frac1{L_k +1} + \cdots +
\frac1k \sim -\log \frac{L_k}{k}$, and define the constant $\alpha_k =
\E[A_k]$.  Let $X_\ell$ be the waiting time for the increasing path to grow
from length $\ell$ to length $\ell+1$.  From what we have shown,
$\E[X_\ell] \sim \frac{\alpha_k}{n-\ell}$. As in the analysis of the greedy
algorithm, $\E[X_\ell^2] = O(\frac1{(n-\ell)^2})$, where the dependence on
the constant $k$ is absorbed into the big-$O$.  This shows that
asymptotically almost surely,
\[
	\sum_{i=1}^\ell X_i \sim \E \left[\sum_{i=1}^\ell X_i \right].
\]
We determine the length of the path by finding the point at which this expected value reaches $1$:
\[
	1 = \sum_{i=1}^\ell \E[X_i] \sim \alpha_k \left(\frac1n + \cdots + \frac1{n-\ell}\right)  \sim \alpha _k\log \frac{n}{n-\ell}.
\]
Therefore the algorithm typically achieves a length $\ell$ such that $\frac
\ell n \sim 1- e^{-1/\alpha_k}$.

\subsection{Determination of constant}

In order to determine the numerical bounds in Theorem \ref{k-greedy}, we
must understand $\alpha_k$.  In~\cite{golomb98}, a recurrence relation is
given for the number of permutations of $\{1, \dots, n\}$ with greatest
cycle length $s$. Using our notation, we present a modified version of this
recurrence: if $L_n$ is the length of the longest cycle in a random
permutation of $\{1, 2, \dots, n\}$, then for $1 \le s \le n$,
\[
	\Pr[L_n = s] = \sum_{j=1}^{\floor{n/s}} \frac{1}{j! \, s^j} \Pr[L_{n-sj} \le s-1],
\]
where $L_0$ is the constant $0$ whenever it occurs. This recurrence allows for an exact numerical
computation of $\alpha_k = \E[A_k]$ for any $k$. Several seconds of
computation are enough to confirm that $\alpha_{100} < 0.523$, which
implies that the 100-greedy algorithm typically finds an increasing path of
length at least $cn$, where $c > 1 - e^{-1/0.523} > 0.85$.

It is natural to wonder whether a particular finite choice of $k$ would be
optimal for the $k$-greedy algorithm.  Using a careful coupling argument,
we can show that $\alpha_k$ is monotone decreasing with respect to $k$: if
we consider $L_k$ and $L_{k+1}$ as stages in the same Chinese restaurant
process, we have
\[
	\E[A_{k+1} - A_k \mid L_k] \le \frac1{k+1} - \frac1{L_k} \cdot \frac{L_k}{k+1} = 0
\]
since with probability at least $\frac{L_k}{k+1}$, the longest cycle increases in length. Therefore as $k \to \infty$, $\alpha_k$ approaches some constant $\alpha$.

Therefore, no finite $k$ is optimal.  Since the Golomb--Dickman constant
$\lim_{k \to \infty} \E[\frac{L_k}{k}] \approx 0.6243$ has no closed form,
we expect the same to be true for $\alpha = \lim_{k \to \infty} \E[-\log
  \frac{L_k}{k}]$.  Our numerical methods estimate $\alpha \approx 0.5219$,
  so our choice of $k=100$ already achieves bounds which are close to
  optimal for $k$-greedy algorithms.

\subsection{Rigorous analysis}
\label{sec:k-greedy-rigorous}

There are two obstacles in the way of the uniformity we assumed for this
analysis. On the one hand, we may expose potential edges, add them to the
search tree, but fail to use them (deleting them from the search tree as we
pass to a subtree), and then encounter these edges again, which increases
the waiting time because these edges can't be added to to the search tree.
On the other hand, when a minimal edge is found, all other edges we
consider gain negative information: their label is not within some range
$[t_1, t_2]$, which only helps us because we have $t \ge t_2$ from that
point on. When these edges are considered a second time, the waiting time
decreases.  In this section, we show that both of these obstacles are
asymptotically irrelevant.  In our discussion, an \emph{exposed}\/ edge is
one whose label has been completely determined.  Even if an edge label has
received negative information, we still call it \emph{unexposed}.

\begin{lemma}
  \label{lemma:asymptotic-uniformity}
  With probability $1-o(1)$, throughout the entire time during which the
  path grew to length $0.99n$, the following conditions held:
  \begin{enumerate}
    \item Each vertex of the graph was incident to at most $o(n)$ exposed
      edges.
    \item For each unexposed edge of the graph, the total length of the
      intervals of negative information was $o(1)$.
  \end{enumerate}
\end{lemma}
\begin{proof}
  First suppose that no vertex has appeared in the search tree for more
  than $O(\log n)$ steps of the algorithm. Then the conclusions of the
  lemma follow:
  \begin{enumerate}
    \item A vertex acquires an exposed edge either as it joins the search
      tree, or when it's already in the search tree and acquires a child,
      and there are $O(\log n)$ such steps. Therefore no vertex has more
      than $O(\log n)$ exposed edges.
    \item An edge acquires negative information only when one of its
      endpoints is in the search tree, which occurs $O(\log n)$ times. Each
      of those times, that endpoint had at least $0.01 n$ edges to vertices
      outside the increasing path, and only $O(\log n)$ of these are
      exposed. Therefore $O(n)$ edges are always available to choose from,
      and the waiting time is $O(1/n)$ in expectation.  Therefore the total
      waiting time the edge observes, which is equal to the negative
      information it acquires, is $o(1)$.
  \end{enumerate}
  When the algorithm begins, it's certainly true that no vertex has
  appeared in the search tree for more than $O(\log n)$ steps. Therefore
  our conclusions initially hold. Together they imply that at every step,
  there are $O(n)$ vertices which could potentially enter the search tree;
  since all edges have $o(1)$ negative information, their probabilities of
  having the next smallest label are asymptotically equal, and so no vertex
  has more than an $O(1/n)$ chance of being chosen. The algorithm runs for
  at most $kn$ steps, so with high probability no vertex enters the search
  tree more than $O(\log n)$ times.

  Once a vertex enters the search tree, it stays there for at most $k^2$
  steps: initially, its level in the tree is at most $k$, and at intervals
  of at most $k$ steps, the search tree is replaced by a subtree and so the
  vertex either leaves the search tree or has its level reduced by $1$.
  Therefore no vertex appears in the search tree for more than $O(\log n)$
  steps, and the conclusions follow.
\end{proof}

Let $\cF_\ell$ be the $\sigma$-algebra generated by the information
revealed at the time the path reaches length $\ell$, and recall that
$X_\ell$ is the waiting time for the increasing path to grow from length
$\ell$ to $\ell+1$, so that $X_\ell$ is $\cF_{\ell+1}$-measurable. Choose $\epsilon > 0$, and let $T
= (1 - e^{-(1-3\epsilon)/\alpha_k})n$.  Define the stopping time $\tau$ to be the
lesser of $T$, or the first $\ell$ for which Lemma~\ref{lemma:asymptotic-uniformity} fails, or for which the waiting time
$X_\ell$ exceeds $\frac{2 \log n}{n}$, or for which the total waiting time
$X_0 + \cdots + X_\ell$ exceeds $1 - \epsilon$.  These conditions are chosen so
that we will have $\tau = T$ a.a.s.; let us assume that for now, and
establish it later.  Define the martingale $(Z_t)$ as follows.  Let $Z_0 =
0$, and for each $t < \tau$, let $Z_{t+1} = Z_t + X_t - \E[X_t \mid
  \cF_t]$.  For each $t \geq \tau$, let $Z_{t+1} = Z_t$.

We next study the successive martingale differences $Z_{\ell+1} - Z_\ell =
X_\ell - \E[X_\ell \mid \cF_\ell]$.  For this, it is helpful to identify
that the most critical information from $\mathcal{F}_\ell$ is the number of
vertices in the search tree at the time the path reaches length $\ell$.  

\begin{lemma}
  \label{lemma:vertex-waiting-time}
  Suppose that $\ell < \tau$ (our stopping time).  Let $A_{\ell, s}$ be the
  event that at the time the path reaches length $\ell$, the search tree
  contains $s$ vertices.  Then $\E[X_\ell \mid \mathcal{F}_\ell, A_{\ell,
  s}] \sim \frac1{n-\ell} \sum_{i=s}^k \frac{1}{i}$.
\end{lemma}
\begin{proof}
  It suffices to show that if the search tree currently contains $i$
  vertices, then the expected waiting time until the search tree contains
  $i+1$ vertices is $(1+o(1)) \frac{1}{i(n-\ell)}$.  To this end, recall
  that in the ideal case, there are $j(n-\ell-i)$ edges to choose from,
  each associated with a $\Uniform(0,1)$ waiting time, and the minimum
  of $i(n-\ell-i)$ waiting times has expected value
  $\frac1{i(n-\ell-i)+1}$, which is $(1+o(1)) \frac{1}{i(n-\ell)}$ since
  $\ell < \tau$ implies that $n - \ell$ is still linear in $n$, while $i
  \leq k$ is constant.

  In reality, the edge labels are not $\Uniform(0,1)$, but if we've
  revealed that an edge's label is not contained in some intervals of total
  length $t$, we can still model the waiting time for that edge as
  $\Uniform(0, 1-t)$: the exact location of the intervals is
  irrelevant, since we will never look at labels with that value anyway.
  Since $\ell < \tau$, any edge we look at has $o(1)$ negative information
  total, so its associated waiting time is uniform on an interval of length
  $1 - o(1)$.  As in the ideal case, the minimum of $i(n-\ell-i - o(n))$
  such random variables has expected value $(1+o(1))\frac1{i(n-\ell)}$,
  and summing over $i$ as the search tree grows from $s$ to $k$ vertices,
  we establish the lemma.
\end{proof}

In light of Lemma~\ref{lemma:vertex-waiting-time}, our stopping time $\tau$
ensures that $(Z_t)$ is Lipschitz with successive differences bounded by $L
= \frac{2 \log n}{n}$.  Hence by the Azuma--Hoeffding inequality, we have
$Z_T \leq \frac{\log^2 n}{\sqrt{n}}$ a.a.s.  Also, since $\tau = T$ a.a.s.,
by unraveling the construction of $(Z_t)$, we see that a.a.s., the total
waiting time satisfies
\begin{displaymath}
  X_0 + X_1 + \cdots + X_T 
  \leq 
  \frac{\log^2 n}{\sqrt{n}}
  +
  \sum_{\ell = 0}^T \E[X_\ell \mid \cF_\ell] ,
\end{displaymath}
where the sum of conditional expectations on the right is itself another
random variable, which we must control.

Let $S_\ell$ be the following $\mathcal{F}_\ell$-measurable random
variable: if our stopping time has not occurred at the time our path
reaches length $\ell$, then let it count the number of vertices in the
search tree.  Otherwise, let it be a completely independent random
variable, distributed as the length of the longest cycle in a uniformly
random permutation of $\{1, \ldots, k\}$.  Define the random variables
\begin{displaymath}
  Y_\ell
  =
  \frac{1}{n-\ell}
  \sum_{i=S_\ell}^k \frac{1}{i} .
\end{displaymath}
By Lemma~\ref{lemma:vertex-waiting-time}, the total waiting time is then
a.a.s.\ at most
\begin{displaymath}
  X_0 + X_1 + \cdots + X_T 
  \leq 
  \frac{\log^2 n}{\sqrt{n}}
  +
  (1+o(1)) \sum_{\ell = 0}^T Y_\ell
  \leq
  \sum_{\ell = 0}^T Y_\ell + \epsilon.
\end{displaymath}
Therefore, it now remains to show that a.a.s.,
\begin{displaymath}
  \sum_{\ell = 0}^T Y_\ell
  \leq
  1-2\epsilon.
\end{displaymath}
This will imply that $X_0 + \cdots + X_T \le 1-\epsilon$ and therefore $\tau=T$.

Again, we have an adapted process in which each $Y_\ell$ is
$\cF_\ell$-measurable, but it helps to study the $Y_\ell$ (or equivalently,
$S_\ell$) with respect to the coarsest possible filtration.  Specifically,
to observe $S_\ell$, we now only need to watch the evolution of the search
tree, and crucially, we may proceed by revealing only the number of
vertices in the subtree of each child of the root.  If we reveal these
numbers at every step when an edge is added to the search tree, then in the
ideal case, each subtree receives the edge with probability proportional to
its size, and we have exactly the Chinese Restaurant Process.  When we
reach $k$ edges and pass to the largest subtree, we reveal the next level
of subtree size information.  In the ideal case, conditioned on the
previous partition and the size of the new search tree, when we reveal the
new partition of subtree sizes, the distribution is precisely a new and
independent Chinese Restaurant Process.  It turns out that reality is not
far off.  Let $\cG_\ell$ be the $\sigma$-algebra generated by $\{S_0,
\ldots, S_\ell\}$, i.e., the natural filtration.

\begin{lemma}
  \label{lemma:chinese-distribution}
  For each $\ell \leq T$, $\Pr[S_\ell = s \mid \cG_{\ell-1}] = (1+o(1))
  \Pr[L_k = s]$, where $L_k$ is the length of the longest cycle in a
  uniformly random permutation of $\{1, \dots, k\}$.
\end{lemma}
\begin{proof}
  If $\ell \ge \tau$, then we have perfect equality because $S_\ell$ is
  then distributed exactly as $L_k$.  Otherwise, all vertices in the search
  tree have $n - \ell - o(n)$ available edges with waiting time uniform on
  an interval of length $1-o(1)$. It follows that up to a factor of
  $1+o(1)$, each vertex in the search tree has approximately the same
  probability of acquiring a child as any other vertex.
  
For any $\ell$, the search tree at the moment when the path reaches length $\ell$ can be described as a recursive tree on $k$ edges: the vertices of the tree are labeled by the order in which they enter the search tree. By the history up to length $\ell$, we mean the sequence of $(R_1, \dots, R_\ell)$ of recursive trees obtained at lengths $1, 2, \dots, \ell$. 

Not every sequence of recursive trees is a valid history: the trees must be consistent, since $R_{i+1}$ must be a suitably relabeled extension of the largest subtree at the root of $R_i$. Nevertheless, if we partition all valid histories by the value of $R_\ell$, there is a natural bijective correspondence between any two parts: to replace $R_\ell$ by $R_\ell'$, we must simply make the same substitution in subtrees of $R_{\ell-1}$, $R_{\ell-2}$, and so on, going back to $R_{\ell-k+1}$ at worst.

Two histories corresponding in this way have the same value of $S_0, S_1, \dots, S_{\ell-1}$: while the shape of the subtrees measured by these random variables may change, the size does not. Moreover, since the two histories agree on all but the last $k$ trees, they only disagree in at most $k^2$ steps of the algorithm, so the probability of obtaining them differs by a factor of $(1+o(1))^{k^2} = 1 + o(1)$. It follows that $(R_\ell \mid \cG_{\ell-1})$ is asymptotically uniformly distributed.

In a uniformly chosen recursive tree, the size of the largest subtree at the root has the same distribution as $L_k$. Therefore $S_\ell$, the size of the largest subtree at the root of $R_\ell$, satisfies $\Pr[S_\ell = s \mid \cG_{\ell-1}] = (1+o(1))
  \Pr[L_k = s]$, as desired.
\end{proof}

Since $k$ is a constant, each $\Pr[L_k = s]$ is a constant, as $s$ ranges
from 1 to $k$.  So, by Lemma~\ref{lemma:chinese-distribution}, the
conditional distribution of $Y_\ell$ given $\cG_{\ell-1}$ is also
supported on $k$ values in the range $\Theta(\frac{1}{n})$, with all
probabilities bounded away from 0 and 1 by at least some constant depending
on $k$.  Crucially, regardless of the particular $\cG_{\ell-1}$, the
distribution of $Y_\ell$ is always asymptotically $\frac{1}{n-\ell} A_k$,
where $A_k = \frac{1}{L_k} + \cdots + \frac{1}{k}$ was the random variable
defined with respect to the longest cycle in the Chinese Restaurant Process
at the end of Section \ref{sec:intuitive}.

To finish the analysis, define the martingale
\begin{displaymath}
  W_\ell = (Y_0 - \E[Y_0])
  + (Y_1 - \E[Y_1 \mid \cG_0])
  + \cdots
  + (Y_\ell - \E[Y_\ell \mid \cG_{\ell-1}]).
\end{displaymath}
It is Lipschitz with successive variations bounded by $O(\frac{1}{n})$
because $k$ is a constant, and so the Azuma--Hoeffding inequality applied to
$(W_\ell)$ implies that a.a.s.,
\begin{displaymath}
  W_T \leq \frac{\log n}{\sqrt{n}},
\end{displaymath}
or equivalently, that
\begin{displaymath}
  \sum_{\ell=0}^T Y_\ell
  \leq
  \frac{\log n}{\sqrt{n}}
  +
  \E[Y_0] + \E[Y_1 \mid \cG_0] + \cdots + \E[Y_T \mid \cG_{T-1}],
\end{displaymath}
where the right hand side is a random variable because of the conditional
expectations.  Now we use the fact that each random variable $\E[Y_\ell
\mid \cG_{\ell-1}]$ is $(1+o(1)) \frac{1}{n-\ell} \E[A_k] = (1+o(1))
\frac{\alpha_k}{n-\ell}$, by definition of $\alpha_k$. Due to our choice of $T$,
we have
\[
	\sum_{\ell=0}^T \frac{\alpha_k}{n-\ell} = (1+o(1)) \alpha_k \log \frac{n}{n-T} = 1 - 3\epsilon + o(1).
\]
Therefore we also obtain that a.a.s., $\sum_{\ell = 0}^T Y_\ell \leq 1 - 3\epsilon + o(1) < 1 - 2\epsilon$, as desired,
and so the $k$-greedy algorithm a.a.s.\ achieves a path of length $T$. Letting $\epsilon \to 0$ with $n$ so that $T = (1 - e^{-1/\alpha_k} + o(1))n$, this completes the proof of Theorem~\ref{k-greedy}.
  \hfill $\Box$

\section{Computing the second moment of $H_n$}

The core of our proof of Theorem \ref{main-theorem} is the second moment
calculation for $H_n$, the random variable which tracks the number of
increasing Hamiltonian paths in a uniformly random edge ordering.  This
second moment, $H_n^2$, counts the number of ordered pairs of increasing
Hamiltonian paths, which can be expressed as a sum of indicator variables:
$H_n^2 = \sum_A \sum_B I_{A,B}$, where $A$ and $B$ range over all
Hamiltonian paths, and $I_{A,B}=1$ if both paths are increasing when $f$ is
chosen, and $0$ otherwise.  Note that although we are working with
undirected graphs, we consider Hamiltonian paths with direction, and
therefore, when we speak of a Hamiltonian path in this section, we are
referring to a permutation of the $n$ vertices.  In particular, each
undirected $n$-vertex path will correspond to two such permutations, and
will appear twice in our indexing, once in each direction.

We begin by grouping the indicator variables into equivalence classes which
we call \emph{intersection profiles}. Two pairs of paths $(A,B)$ and
$(A',B')$ are in the same intersection profile if there is a permutation of
the edges of $K_n$ (without necessarily preserving all pairwise incidence
relations between edges) that takes the paths $A$ and $B$ to $A'$ and $B'$.
We can represent such a profile as a graph by separating the vertices of
$A$ and $B$ that are not endpoints of a common edge, as in
Figure~\ref{fig:profile}.

There are many ways to order the edges of an intersection profile to make
both paths increasing: only the relative orders of the edges common to both
$A$ and $B$ are fixed.  In our counting, we further split each intersection
profile up into labeled profiles, in which one such ordering is chosen
(again, see Figure~\ref{fig:profile} for an example).
\begin{figure}
	\centering
	\includegraphics[width=0.7\textwidth]{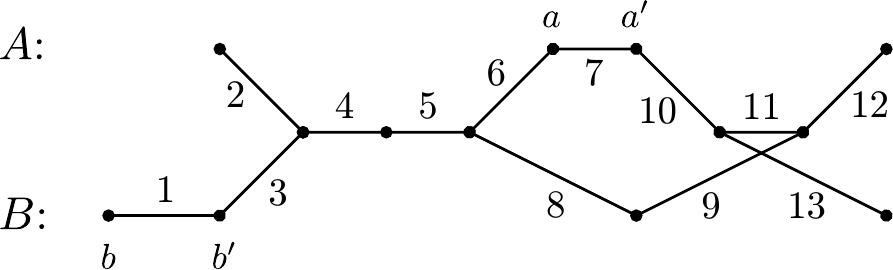}
	\caption{A labeled profile for two paths of length $8$, with $c=3$,
        $k=2$, $\ell=1$.}
	\label{fig:profile}
\end{figure}
We keep track of three parameters of an unlabeled profile $P$:
\begin{enumerate}
\item $c(P)$ is the number of common edges shared by the two paths $A$ and $B$.
\item $k(P)$ is the number of common segments shared by the two paths: this satisfies $k(P) \le c(P)$ because each common segment contains at least one edge, and $k(P) \le n-c(P)$ because each path contains at least one edge between two common segments.
\item $\ell(P)$ is the number of common segments which consist of exactly
  one edge: since $c(P) \geq \ell(P) + 2[k(P) - \ell(P)]$, this satisfies
  $\ell(P) \ge 2k(P) - c(P)$.
\end{enumerate}

Let $\cP(c,k,\ell)$ be the set of profiles $P$ such that $c(P)=c, k(P)=k,
\ell(P)=\ell$, and $\cL(c,k,\ell)$ the corresponding set of labeled
profiles.  The total number of edges in $P$ is $2(n-1) - c$; therefore, if
paths $A$ and $B$ fit some intersection profile $P$, and the edge ordering
$f$ is chosen randomly, the ordering of $A$ and $B$ will match any given
labeled version of $P$ with probability $\frac1{(2n-c-2)!}$. So we can write
\begin{equation}
\label{eq:second-moment}
\E[H_n^2] = \sum_{c,k,\ell} \sum_{P \in \cL(c,k,\ell)} \frac{|P|}{(2n-c-2)!}
\end{equation}
where $|P|$ is the number of pairs of paths $(A,B)$ that fit the unlabeled version of $P$.
We split the sum \eqref{eq:second-moment} into several parts:
\begin{enumerate}
\item $S_1$, the sum over $c \le \log n$ (most other notions of ``small'' $c$ would also be sufficient here);
\item $S_2$, the sum over $\log n < c \le \frac{9}{10}n$; and
\item $S_3$, the sum over $c > \frac{9}{10}n$.
\end{enumerate}

It will be $S_1$ that contributes the most to $\E[H_n^2]$, and so we state
two lemmas that provide asymptotically exact estimates (with multiplicative
error tending to zero as $n \rightarrow \infty$) for $S_1$ while only
serving as rough upper bounds in $S_2$ and $S_3$.  Here, ${n \choose a, b,
c}$ denotes the multinomial coefficient $\frac{n!}{a! b! c!}$.

\begin{lemma}
\label{lemma:profiles}
For all $c, k, \ell$, 
\[
	|\cL(c,k,\ell)| \le 2^\ell {k \choose \ell} {c-k-1 \choose k-\ell-1} {2(n-c-1)+k \choose n-c-1, n-c-1, k}
\]
and $|\cL(c,k,\ell)|$ is asymptotically equal to the right hand side when
$c \le \log n$. Furthermore,
\[
	\sum_{\ell} |\cL(c,k,\ell)| \le 2^k {c-1 \choose k-1} {2(n-c-1)+k \choose n-c-1, n-c-1, k}
\]
\end{lemma}

\begin{lemma}
\label{lemma:vertices}
For all $P \in \cP(c, k, \ell)$, $|P| \le n! (n-c-k)!$; furthermore, if $c \le \log n$, then $|P| \sim e^{-2} n! (n-c-k)!.$
\end{lemma}

By using these lemmas, we can write out algebraic expressions for $S_1$, $S_2$, and $S_3$:
\begin{align}
S_1 &\sim e^{-2} \sum_{c=0}^{\log n} \sum_{k,\ell} 2^\ell {k \choose \ell} {c-k-1 \choose k-\ell-1} {2(n-c-1)+k \choose n-c-1, n-c-1, k} \frac{n! (n-c-k)!}{(2n-c-2)!} \label{eq:s1} \\
S_2 &\le \sum_{c=\log n}^{9n/10} \sum_k 2^k {c-1 \choose k-1} {2(n-c-1)+k \choose n-c-1, n-c-1, k} \frac{n! (n-c-k)!}{(2n-c-2)!} \label{eq:s2} \\
S_3 &\le \sum_{c=9n/10}^{n-1} \sum_k 2^k {c-1 \choose k-1} {2(n-c-1)+k \choose n-c-1, n-c-1, k} \frac{n! (n-c-k)!}{(2n-c-2)!}. \label{eq:s3}
\end{align}

Therefore, $\E[H_n^2] \sim e n^2$ will follow from:

\begin{lemma}
  The right hand side of the expression \eqref{eq:s1} is asymptotic to $e
  n^2$, and both of the right hand sides of \eqref{eq:s2} and \eqref{eq:s3}
  simplify to $o(n^2)$.
  \label{lem:s123-asymp}
\end{lemma}

\subsection{\boldmath Asymptotics for $|\cL|$ and $|P|$}

\begin{proof}[Proof of Lemma~\ref{lemma:profiles}]
Let $c$, $k$, and $\ell$ be given.  For the remainder of the proof, let $m
= n-c-1$ stand for the number of edges that belong to $A$ but not $B$
(equivalently, to $B$ but not $A$), when paths $A$ and $B$ fit a profile
from $\cL(c,k,\ell)$.  Consider the following two-stage method for
selecting a labeled profile from $\cL(c,k,\ell)$.  All such labeled
profiles will be reachable in this way.
\begin{enumerate}
\item Choose the sequence of lengths for the common segments, and their
  relative orientations within paths $A$ and $B$. The segments of length
  $1$ can appear in ${k \choose \ell}$ positions, and the remaining
  $c-\ell$ common edges can be partitioned into $k-\ell$ ordered parts of
  size at least $2$ in ${c-k-1 \choose k-\ell-1}$ ways.  Common segments of
  length at least 2 will already have a fixed orientation, because their
  sequential edge labels will need to be increasing with respect to both
  $A$ and $B$.  On the other hand, a common segment of length 1 could be
  traversed in either the same direction by both paths, or in opposite
  directions (as in the case of the second common segment in
  Figure~\ref{fig:profile}).  Therefore, by considering relative
  orientations, we gain another factor of exactly $2^\ell$.

\item Now that the sequence of lengths and directions has been fixed for
  all common segments, it remains to choose an order in which the $k$
  common segments, the $m$ edges of $A$, and the $m$ edges of $B$ appear.
  For this, we construct labeled profiles from strings of $m$ \emph{A}'s,
  $m$ \emph{B}'s, and $k$ \emph{C}'s.  For example, if we have already
  fixed the first common segment to have length 2, and the other common
  segment to have length 1, traversed in both directions, then the string
  \emph{BABCAABBACAB}\/ corresponds precisely to Figure~\ref{fig:profile}.
  There are at most ${2m+k \choose m, m, k}$ such strings of \emph{A}'s,
  \emph{B}'s, and \emph{C}'s.
\end{enumerate}

The above two-step procedure immediately implies the claimed upper bound on
$|\cL(c, k, \ell)|$ in Lemma~\ref{lemma:profiles}.  Our next objective is
to show that this bound is asymptotically correct for $c \leq \log n$.  The
second step overestimates $|\cL(c,k,\ell)|$ because of two possible illegal
interactions between adjacent common segments: (a) we cannot have two
consecutive \emph{C}'s, separated by all \emph{A}'s or all \emph{B}'s, and
(b) no consecutive \emph{C}'s can be separated by exactly one
\emph{A}\/ and exactly one \emph{B}.  We will show that the number of such
strings is $o(1)$-fraction of the total number of strings which appear in
the second step.

First, we control the number of strings which have two \emph{C}'s which are
separated only by \emph{B}'s.  For this, fix one of the $k-1$ gaps between
common segments, and suppose that there are exactly $m-i$ \emph{B}'s in the
gap, with $0 \le i \le m$.  Then, those strings are in bijective
correspondence with the strings with exactly $m$ \emph{A}'s, exactly $i$
\emph{B}'s, and exactly $k-1$ \emph{C}'s: the bijection is realized by the
deletion of a segment \emph{BB\dots BC}\/ (with $m-i$ \emph{B}'s) after the
\emph{C}\/ which corresponds to the beginning of the gap.  Thus, the total
number of strings which have two \emph{C}'s separated only by \emph{B}'s is
at most
\begin{align}
  \nonumber
  (k-1) \sum_{i=0}^{m} {m + i + k-1 \choose m, i, k-1} 
  &= (k-1) {m+k-1 \choose k-1} \sum_{i=0}^m {m+k-1+i \choose i} \\
  \nonumber
  &= (k-1) {m+k-1 \choose k-1} {2m+k \choose m} \\
  \label{eq:CBBC}
  &= (k-1) \frac{k}{m+k} {2m+k \choose m, m, k}.
\end{align}

The same bound applies for the total number of strings with two \emph{C}'s
separated only by \emph{A}'s.  Similarly, if two \emph{C}'s are separated
by exactly one \emph{A}\/ and one \emph{B}, there are two possible
orderings (\emph{AB}, \emph{BA}) between the \emph{C}'s, and so the total
number of such strings is at most
\begin{align}
  \nonumber
  (k-1) 2 {2m+k - 3 \choose m-1, m-1, k-1}
  &=
  (k-1) 2 \frac{(2m+k-3)!}{(m-1)! (m-1)! (k-1)!} \\
  \label{eq:CABC}
  &=
  \frac{2 (k-1) m^2 k}{(2m+k) (2m+k-1) (2m+k-2)} \cdot \frac{(2m+k)!}{m! m! k!} 
\end{align}

When $k \le c \le \log n$, both \eqref{eq:CBBC} and \eqref{eq:CABC} are of
order $\frac{k^2}{m} {2m+k \choose m, m, k}$, and $\frac{k^2}{m} = o(1)$.
Therefore the true number of choices to be made in the second step is
indeed $(1-o(1)){2m+k \choose m, m, k}$ for small $c$ and $k$, as claimed.
Finally, to obtain the rougher approximation for the second part of the
lemma, we forget about the value of $\ell$, and get an upper bound by
assuming that all $k$ segments can be reversed, even if they don't have
length $1$.  Effectively, we use $2^\ell \leq 2^k$ and $\sum_\ell
\binom{k}{\ell} \binom{c-k-1}{k-\ell-1} = \binom{c-1}{k-1}$.
\end{proof}

\begin{proof}[Proof of Lemma~\ref{lemma:vertices}]
The first part of the lemma is immediate: there are $n!$ ways to choose the
$n$ vertices of $A$, and $(n-c-k)!$ ways to choose the remaining $n-c-k$
vertices of $B$. However, this mistakenly counts some pairs of paths that
don't fit the profile $P$.  For example, if the profile in Figure
\ref{fig:profile} embeds into $K_n$ by sending the vertices labeled $a$ and
$b$ to the same vertex in $K_n$, and sending the vertices labeled $a'$ and
$b'$ to the same vertex in $K_n$, then the embedded paths no longer
correspond to the intersection profile in the figure, because additional
common segments have been created.  So, to prove the second part of the
lemma, we must estimate the probability that in a random permutation of the
$n-c-k$ vertices of $B$ which are not on common segments with $A$, no new
common segments are created between the embedded paths.

We first consider the case $c=k=0$, which clearly corresponds to the
probability that in a random permutation of $\{1,2,\dots,n\}$, no two
consecutive elements are adjacent. Wolfowitz~\cite{wolfowitz44} has shown
that asymptotically, the number of adjacent consecutive elements has the
Poisson distribution with mean $2$, and therefore we obtain the desired
probability of $e^{-2}$.

For the general case, suppose that the $n$ vertices of $A$ have been fully
embedded into the $K_n$.  Exactly $n-c-k$ of them correspond to vertices of
$A$ which are not shared by $B$ in the profile diagram.  Following the
natural order for $A$ in the profile, label those $n-c-k$ embedded vertices
(in $K_n$) by $1, 2, \ldots, n-c-k$.  Then, each embedding of the remaining
$n-c-k$ vertices of $B$ (which completes the embedding of the two paths $A$
and $B$) corresponds precisely to a distinct permutation of $\{1, 2,
\ldots, n-c-k\}$, because both $A$ and $B$ are Hamiltonian paths, and thus
each use all of the vertices.  Here, the permutation is the order in which
the vertices $\{1, 2, \ldots, n-c-k\}$ are visited when the embedded $B$ is
traversed in its natural order.  Permutations with adjacent consecutive
elements still approximately correspond to embeddings which create
extraneous common segments, and it remains to quantify the error in the
approximation, which arises at junctions with common segments.

When $A$ is traced in its natural order, there are either $k-1$ or $k$
vertices of $A$ which come immediately before the start of a common
segment.  (There are $k-1$ if the first vertex of $A$ is already part of a
common edge between the two paths, and $k-1$ otherwise.)  Let $i_j \in \{1,
2, \ldots, n\}$ be the label in $K_n$ of the embedded vertex corresponding
to the vertex of $A$ which comes immediately before the start of the $j$-th
common segment.  If there are only $k-1$ such vertices, then leave $i_1$
undefined, and ignore all references to it in the remainder of this
argument.

In terms of the $i_j$'s, permutations $\sigma$ with adjacent consecutive
elements correspond to embeddings with extraneous common segments, except
when for some $j$, we have that (a) $i_j$ and $i_{j}+1$ are adjacent in
$\sigma$, or (b) the vertex of $B$ which immediately precedes the $j$-th
common segment maps to $i_j$, or (c) the vertex of $B$ which comes right
after the $j$-th common segment maps to $i_j + 1$.  To see this, observe
that (a) identifies a ``false positive,'' in which the elements are
adjacent in $\sigma$, but are actually separated by a common segment in
$B$'s traversal.  On the other hand, (b) and (c) represent the ``false
negatives,'' in which the $j$-th common segment is unduly extended.
Fortunately, a union bound over all $j$ shows that the probability of (a)
happening is at most $\frac{2k}{n-c-k-1}$, the probability of (b) happening
is at most $\frac{k}{n-c-k}$, and the probability of (c) has the same
bound.  All of these quantities are $o(1)$ for $k \leq c \leq \log n$, and
therefore the probability that no new common segments are created differs
by $o(1)$ from the probability that no consecutive elements occur, and is
also asymptotically $e^{-2}$.
\end{proof}

\subsection{\boldmath Estimating $S_1$}

To simplify the expressions involved, we use the notation $(n)_k$ for the
falling power $\frac{n!}{(n-k)!} = n(n-1)\cdots(n-k+1)$. This satisfies
$(n)_k \sim n^k$ for $k = o(\sqrt{n})$. In particular, for $c \le \log n$,
we have $k \le \log n$ as well; therefore for falling powers linear in $c$
and $k$, we may freely use this asymptotic relation.  Starting from
\eqref{eq:s1}, we have:
\begin{align*}
S_1 &\sim e^{-2} \sum_{c=0}^{\log n} \sum_{k,\ell} 2^\ell {k \choose \ell} {c-k-1 \choose k-\ell-1} {2(n-c-1)+k \choose n-c-1, n-c-1, k} \frac{n! (n-c-k)!}{(2n-c-2)!} \\
&= e^{-2} \sum_{c=0}^{\log n} \sum_{k,\ell} 2^\ell {k \choose \ell} {c-k-1 \choose k-\ell-1} \frac{(2n-2c+k-2)! n! (n-c-k)!}{k! (n-c-1)!^2 (2n-c-2)!} \\
&= e^{-2} \sum_{c=0}^{\log n} \sum_{k,\ell} 2^\ell {k \choose \ell} {c-k-1 \choose k-\ell-1} \frac1{k!} \cdot 
\frac{(n)_{c+1}}{(n-c-1)_{k-1} (2n-c-2)_{c-k}} \\
&\sim e^{-2} n^2  \sum_{c=0}^{\log n} \sum_{k,\ell} {k \choose \ell} {c-k-1 \choose k-\ell-1} \frac{2^{\ell-c+k}}{k!}.
\end{align*}
Splitting off the $e^{-2} n^2$ factor, it turns out that the remaining
double sum converges to a constant $C$ (which depends only on $c$, $k$, and
$\ell$, not $n$) as $n \rightarrow \infty$, and we now compute this limit.
Recall that the constraints on $k$ and $\ell$ in the inner sum are that $0
\le \ell \le k \le c$ and, furthermore, that $c \ge \ell + 2(k-\ell) =
2k-\ell$ (which is a stronger bound than $c \geq k$). Therefore
\begin{align*}
  C 	&= \sum_{c=0}^\infty \sum_{k,\ell} {k \choose \ell} {c-k-1 \choose k-\ell-1} \frac{2^{\ell-c+k}}{k!} \\
  &= \sum_{k=0}^\infty \frac{2^k}{k!} \sum_{\ell=0}^k 2^\ell {k \choose \ell} \sum_{c=2k-\ell}^\infty {c-k-1 \choose k-\ell-1} 2^{-c} \\
  &= \sum_{k=0}^\infty \frac{2^k}{k!} \sum_{\ell=0}^k 2^\ell {k \choose
  \ell} 2^{-2k+\ell} \sum_{j=0}^\infty {j + (k-\ell-1) \choose k-\ell-1}
  2^{-j} \,,
\end{align*}
where we re-parameterized the final sum as $j = c-(2k-\ell)$.  The final
summation is now conveniently in the form of the following power series
identity:
\begin{displaymath}
  \sum_{j=0}^\infty {j+m-1 \choose m-1} z^j = \frac1{(1-z)^m} \,.
\end{displaymath}
Therefore,
\begin{align*}
  C &= \sum_{k=0}^\infty \frac{2^k}{k!} \sum_{\ell=0}^k 2^\ell {k \choose \ell} 2^{-2k+\ell} \cdot 2^{k-\ell}	 \\
  &= \sum_{k=0}^\infty \frac{1}{k!} \sum_{\ell=0}^k 2^{\ell} {k \choose
  \ell}  = \sum_{k=0}^\infty \frac{3^k}{k!} = e^3 \,,
\end{align*}
which implies that $S_1 \sim e^{-2} n^2 C = e n^2$, as claimed.

\subsection{\boldmath Estimating $S_2$} 

Let $a_{c,k}$ be one of the summands in $\eqref{eq:s2}$. Then
\begin{equation}
	\label{eq:ratio}
	\frac{a_{c,k+1}}{a_{c,k}} = \frac{2 (c-k)}{k (k+1)} \cdot \frac{2n-2c+k-1}{n-c-k}.
\end{equation}
Our goal is to simplify the upper bound on $S_2$ by selecting, for each
$c$, the $k$ that maximizes $a_{c,k}$, and then using this maximum in place
of all the terms with that value of $c$.

First consider $k$ such that $k \le \frac12(n-c)$. In this case, the second
factor of \eqref{eq:ratio} is bounded between $1$ and $5$: on the one hand,
$(2n-2c+k-1) - (n-c-k) = n-c+2k-1 \ge 0$, and on the other hand,
$(2n-2c+k-1) - 5(n-c-k) = 6k - 3(n-c) - 1 \le 0$. Therefore we have
\[
  \frac{2(c-k)}{k(k+1)} \le \frac{a_{c,k+1}}{a_{c,k}} \le \frac{10(c-k)}{k(k+1)}.
\]
If $k$ maximizes $a_{c,k}$ and lies in this range, then $2(c-k) \le k(k+1)$
and therefore $k \ge (1 - o(1)) \sqrt{2c}$; on the other hand, $10(c-k-1)
\ge k(k-1)$ and therefore $k \le (1 - o(1)) \sqrt{10c}$. We may safely and
concisely say $\sqrt{c} < k < 4\sqrt{c}$.

On the other hand, if $k \ge \frac12(n-c)$, since $S_2$ only runs $c$ up to
$\frac{9n}{10}$, we have $k \ge \frac12(n-c) \ge n/20$ in the denominator
of~\eqref{eq:ratio}, and we always have $k \le c \le n$ in the numerator.
So,
\[
  \frac{a_{c,k+1}}{a_{c,k}} 
  < 
  \frac{2n}{k^2} \cdot \frac{3n}{n-c-k} 
  \le 
  \frac{2n}{n^2/400} \cdot \frac{3n}{n-c-k} 
  = 
  \frac{2400}{n-c-k} \,,
\]
which is less than 1 as long as $k \leq n-c-2400$.  Therefore the
maximizing $k$ is either in the range found above, or between $n-c-2400$
and $n-c$ (since $k \leq n-c$ always).

Let $S_2'$ be the result of replacing in $S_2$ all terms $a_{c,k}$ by
$a_{c,k^*}$ where $\sqrt{c} < k^*(c) < 4 \sqrt{c}$ is the maximizing
$k$ from the range $0 \leq k \leq n-c-2400$.  Then
\begin{align*}
  S_2' &< \sum_{c=\log n}^{9n/10} c \cdot 2^{k^*} {c-1 \choose k^*-1} {2(n-c-1)+k^* \choose n-c-1, n-c-1, k^*} \frac{n! (n-c-k^*)!}{(2n-c-2)!} \\
  &= \sum_{c=\log n}^{9n/10} \frac{k^* 2^{k^*}}{k^*!} {c \choose k^*} \frac{(n)_{c+1}}{(n-c-1)_{k^*-1} (2n-c-2)_{c-k^*}} \\
  &< \sum_{c=\log n}^{9n/10} k^* \left(\frac{2e}{k^*}\right)^{k^*} \left(\frac{ce}{k^*}\right)^{k^*} \frac{(n)_{c+1}}{(n-c-1)_{k^*-1} (2n-c-2)_{c-k^*}} \\
  &< n^2 \sum_{c=\log n}^{9n/10} k^* (2e^2)^{k^*}  \frac{(n-2)_{c-k^*}}{(2n-c-2)_{c-k^*}} \cdot \frac{(n-c+k^*-2)_{k^*-1}}{(n-c-1)_{k^*-1}}.
\end{align*}
We now eliminate the powers of $n$ in the summand. In the first fraction,
since $c \le \frac{9}{10}n$,  $2n-c \ge \frac{11}{10}n$, and therefore
$\frac{(n-2)_{c-k^*}}{(2n-c-2)_{c-k^*}} \le
\left(\frac{10}{11}\right)^{c-k^*}$.  In the second fraction, since $k^*$
is easily less than $\frac12(n-c)$,  $n-c+k^*-2 < \frac32(n-c-1)$, and
therefore $\frac{(n-c+k^*-2)_{k^*-1}}{(n-c-1)_{k^*-1}} < (3/2)^{k^*}$.
Thus
\begin{align*}
  S_2' 	
  &<
  n^2 \sum_{c = \log n}^{9n/10} 
  k^* \left(
  2e^2 \cdot \frac{11}{10} \cdot \frac{3}{2}
  \right)^{k^*}
  \left(\frac{10}{11}\right)^c \\
  &=
  n^2 \sum_{c = \log n}^{9n/10} 
  k^* \left(\frac{33e^2}{10}\right)^{k^*} \left(\frac{10}{11}\right)^c \\
  &< 
  n^2 \sum_{c = \log n}^\infty 
  4\sqrt{c} \left( 
  \frac{33e^2}{10} \cdot \left(\frac{10}{11}
  \right)^{\sqrt{c}/4}\right)^{4\sqrt{c}}.
\end{align*}
The sum is the tail of a convergent series in $c$, and therefore $S_2' =
o(n^2)$.

To show that $S_2 = o(n^2)$, it remains to consider the terms $a_{c,k}$ for
which $n-c-k \le 2400$, as these are potentially not dominated by
$a_{c,k^*}$. For this case, we consider a second ratio:
\[
  \frac{a_{c,k}}{a_{c+1,k}} 
  = 
  \frac{c-k+1}{c} \cdot
  \frac{(2n-2c+k-2)(2n-2c+k-3)}{(n-c+1)(n-c-1)} \cdot
  \frac{n-c-k}{2n-c-2} \,.
\]
Here, $n-c-k \leq 2400$ and $c-k+1 \le c$; all other factors are
$\Theta(n)$ because $c \leq \frac{9n}{10}$, and so the overall ratio is
$O(n^{-1})$. Therefore, once $n$ surpasses some absolute constant, all of
these $a_{c,k}$ with $n-c-k \leq 2400$ satisfy $a_{c,k} \leq a_{n-k, k}$,
and there are at most $2400 n$ of them.  It remains to control $a_{n-k, k}$
in the range $k \geq n-c-2400 \geq (1-o(1)) \frac{n}{10}$, where we used $c
\leq \frac{9n}{10}$.  For those, we have
\begin{align*}
  a_{n-k, k}
  &<
  2^k \binom{n-k-1}{k-1} {3k-2 \choose k-1,k-1,k} \frac{n!}{(n+k-2)!} \\
  &<
  2^k 2^{n-k-1} 3^{3k-2} \cdot \frac{1}{(n+k-2)_{k-2}} \\
  &<
  \frac{54^n}{(n+k-2)_{k-2}} \\
  &<
  \frac{54^n}{n^{(1-o(1))n/10}} 
  =
  o(1) \,.
\end{align*}
Therefore, the total contribution of these residual $a_{c,k}$ is at most
$2400 n \cdot o(1)$, and $S_2 = S_2' + o(n) = o(n^2)$, as claimed.

\subsection{\boldmath Estimating $S_3$}

Let $d = n-c$ and consider the sum for $1 \le d \le \frac{n}{10}$. Then
from \eqref{eq:s3} we get
\begin{align*}
S_3 &\le \sum_{k=1}^{n/10} \sum_{d=k}^{n/10}  2^k {n-d-1 \choose k-1} {2d+k-2 \choose d-1, d-1, k} \frac{n! (d-k)!}{(n+d-2)!} \\
	&\le \sum_{k=1}^{n/10} \sum_{d=k}^{n/10} \frac{2^k}{(k-1)!} \cdot \frac{(n-d-1)!}{(n-d-k)!} \cdot 3^{2d+k-2} \cdot \frac{n! (d-k)!}{(n+d-2)!} \\
	&\le \sum_{k=1}^{n/10} \sum_{d=k}^{n/10} \frac{6^k \cdot 9^{d-1}}{(k-1)!}\cdot \frac{(n-d-1)_{k-1} (d-k)!}{ (n+d-2)_{d-2}} \\
	&\le \sum_{k=1}^{n/10} \sum_{d=k}^{n/10} \frac{6^k \cdot 9^{d-1} \cdot (d-k)!}{(k-1)!\cdot n^{d-k-1}}. 
\end{align*}
Let $b_{d,k}$ be a term of this sum; since $d-k \le d \le \frac{n}{10}$, 
\begin{displaymath}
  \frac{b_{d,k}}{b_{d-1,k}} 
  = 
  \frac{9(d-k)}{n} \le \frac{9}{10}.
\end{displaymath}
Therefore an upper bound on $S_3$ is:
\[
	S_3	\le \sum_{k=1}^{n/10} b_{k,k} \sum_{d=k}^{n/10} \left(\frac{9}{10}\right)^{d-k}
		\le 10 \sum_{k=1}^\infty \frac{6^k \cdot 9^{k-1}}{(k-1)!\cdot n^{-1}} 
		= 60 e^{54} n = o(n^2) \,,
\]
which completes our proof.

\section*{Acknowledgments}

The authors would like to thank Jian Ding for fruitful discussions which
led to the origin of this project.


\begin{thebibliography}{10}

\bibitem{bollobas85}
B.~Bollob{\'a}s.
\newblock {\em Random graphs}.
\newblock Academic Press, London, 1985.

\bibitem{calderbank84}
A.~R. Calderbank, F.~R.~K. Chung, and D.~G. Sturtevant.
\newblock Increasing sequences with nonzero block sums and increasing paths in
  edge-ordered graphs.
\newblock {\em Discrete Math.}, 50(1):15--28, 1984.

\bibitem{chvatal71}
V.~Chv{\'a}tal and J.~Koml{\'o}s.
\newblock Some combinatorial theorems on monotonicity.
\newblock {\em Canad. Math. Bull.}, 14:151--157, 1971.

\bibitem{CooperFrieze}
C.~Cooper and A.~Frieze.
\newblock On the number of Hamilton cycles in a random graph. 
\newblock {\em Journal of Graph Theory}, 13(6):719--735, 1989.

\bibitem{erdos35}
P.~Erd\H{o}s and G.~Szekeres.
\newblock A combinatorial problem in geometry.
\newblock {\em Compositio Math.}, 2:463--470, 1935.

\bibitem{FoxPachSudakovSuk}
  J.~Fox, J.~Pach, B.~Sudakov, and A.~Suk.
  \newblock Erd\H{o}s--Szekeres-type theorems for monotone paths and convex
  bodies.
  \newblock {\em Proceedings of the London Mathematical Society},
  105:953--982 (2012).

\bibitem{GlebovKrivelevich}
R.~Glebov and M.~Krivelevich. 
\newblock On the number of Hamilton cycles in sparse random graphs. 
\newblock {\em SIAM Journal on Discrete Mathematics}, 27(1):27--42, 2013.

\bibitem{golomb98}
S.~W. Golomb and P. Gaal.
\newblock On the number of permutations of {$n$} objects with greatest cycle
  length {$k$}.
\newblock {\em Adv. in Appl. Math.}, 20(1):98--107, 1998.

\bibitem{graham73}
R.~L. Graham and D.~J. Kleitman.
\newblock Increasing paths in edge ordered graphs.
\newblock {\em Period. Math. Hungar.}, 3:141--148, 1973.
\newblock Collection of articles dedicated to the memory of Alfr{\'e}d
  R{\'e}nyi, II.

\bibitem{janson00}
S.~Janson, T.~{\L}uczak, and A.~Ruci{\'n}ski.
\newblock {\em Random Graphs}.
\newblock Wiley, New York, 2000.

\bibitem{Kalmanson}
  K.~Kalmanson.
  \newblock On a theorem of Erd\H{o}s and Szekeres.
  \newblock {\em Journal of Combinatorial Theory, Series A},
  15(3):343--346, 1973.

\bibitem{logan77}
B.~F. Logan and L.~A. Shepp.
\newblock A variational problem for random {Y}oung tableaux.
\newblock {\em Advances in Math.}, 26(2):206--222, 1977.

\bibitem{MoshkovitzShapira}
  G.~Moshkovitz and A.~Shapira.
  \newblock Ramsey theory, integer partitions and a new proof of the
  Erd\H{o}s--Szekeres Theorem.
  \newblock submitted.

\bibitem{robinson94}
R.~W. Robinson and N.~C. Wormald.
\newblock Almost all regular graphs are {H}amiltonian.
\newblock {\em Random Structures Algorithms}, 5(2):363--374, 1994.

\bibitem{steele95}
M.~Steele.
\newblock Variations on the monotone subsequence theme of {E}rd\H{o}s and
  {S}zekeres.
\newblock In {\em Discrete probability and algorithms}, pages 111--131.
  Springer, 1995.

\bibitem{SzaboTardos}
  T.~Szab\'o and G.~Tardos.
  \newblock A multidimensional generalization of the Erd\H{o}s--Szekeres lemma
  on monotone subsequences.
  \newblock {\em Combinatorics, Probability and Computing}, 10(6):557--565,
  2001.

\bibitem{versik77}
A.~M. Ver{\v{s}}ik and S.~V. Kerov.
\newblock Asymptotic behavior of the {P}lancherel measure of the symmetric
  group and the limit form of {Y}oung tableaux.
\newblock {\em Dokl. Akad. Nauk SSSR}, 233(6):1024--1027, 1977.

\bibitem{winkler08b}
P.~Winkler.
\newblock Puzzled: Solutions and sources.
\newblock {\em Commun. ACM}, 51(9):103--103, September 2008.

\bibitem{wolfowitz44}
J.~Wolfowitz.
\newblock Note on runs of consecutive elements.
\newblock {\em Ann. Math. Statistics}, 15:97--98, 1944.

\end{thebibliography}
\end{document}